\documentclass[12pt]{amsart}
\usepackage{amsfonts, amssymb, latexsym, hyperref, xcolor, tikz,tikz-cd,transparent}
\usepackage{ytableau}
\usepackage{empheq}

\newtheorem{theorem}{Theorem}[section]
\newtheorem{proposition}[theorem]{Proposition}
\newtheorem{lemma}[theorem]{Lemma}

\newtheorem*{claim*}{Claim}
\newtheorem{corollary}[theorem]{Corollary}
\newtheorem{Main Conjecture}[theorem]{Main Conjecture}
\newtheorem{conjecture}[theorem]{Conjecture}
\newtheorem{problem}[theorem]{Problem}
\theoremstyle{definition}
\newtheorem{definition}[theorem]{Definition}

\theoremstyle{remark}

\newtheorem{example}[theorem]{Example}

\theoremstyle{plain}

\newcommand\complexes{{\mathbb C}}
\newcommand\integers{{\mathbb Z}}
\newcommand\naturals{{\mathbb N}}

\newcommand\init{{\mathrm{init}}}

\newcommand{\cellsize}{12}
\newlength{\cellsz} 
\newsavebox{\cell}
\sbox{\cell}{\begin{picture}(\cellsize,\cellsize)
\put(0,0){\line(1,0){\cellsize}}
\put(0,0){\line(0,1){\cellsize}}
\put(\cellsize,0){\line(0,1){\cellsize}}
\put(0,\cellsize){\line(1,0){\cellsize}}
\end{picture}}
\newcommand\cellify[1]{\def\thearg{#1}\def\nothing{}%
\ifx\thearg\nothing
\vrule width0pt height\cellsz depth0pt\else
\hbox to 0pt{\usebox{\cell} \hss}\fi%
\vbox to \cellsz{
\vss
\hbox to \cellsz{\hss$#1$\hss}
\vss}}
\newcommand\tableau[1]{\vtop{\let\\\cr
\baselineskip -16000pt \lineskiplimit 16000pt \lineskip 0pt
\ialign{&\cellify{##}\cr#1\crcr}}}

\newcommand{\excise}[1]{}

\begin{document}
\pagestyle{plain}
\title{Combinatorial Commutative Algebra Rules}
\author{Ada Stelzer}
\author{Alexander Yong}
\address{Dept.~of Mathematics, U.~Illinois at Urbana-Champaign, Urbana, IL 61801, USA} 
\email{astelzer@illinois.edu, ayong@illinois.edu}
\date{June 1, 2023}

\begin{abstract}
An algorithm is presented that generates sets of size equal to the degree of a given variety defined by a homogeneous ideal. This algorithm suggests a versatile framework to study various problems in combinatorial algebraic geometry and related fields.
\end{abstract}
\maketitle

\vspace{-0.1in}
\section{Introduction}

Consider a homogeneous ideal $I$ in the polynomial ring $R=\mathbb{C}[x_1,\ldots,x_N]$ defining an algebraic set $V(I)$. The Hilbert polynomial $h_{R/I}(t)$ of $R/I$ captures the dimensions of $(R/I)_t$ for sufficiently large $t$. The \emph{degree} of $V(I)$, denoted $\deg V(I)$, is defined as the leading coefficient of $h_{R/I}(t)$ multiplied by $(\dim V(I))!$. The degree is a positive integer (see, e.g., \cite{Harris}). Numerous combinatorial problems in algebraic geometry and Lie theory seek ``combinatorial rules'' for determining $\deg V(I)$ for various families of ideals $I$.

In general, this report \emph{algorithmically} constructs sets ${\mathcal T}$ whose \emph{unweighted} count equals $\deg V(I)$. In many examples of interest, it indicates existence of a combinatorial rule of the specified form and produces objects called \emph{hieroglyphic tablets} for concrete examination.

\subsection{The algorithm}
The recipe proposed here combines three commutative algebra methods, namely, Gr\"obner bases, polarization, 
and prime decomposition,
as follows:
\begin{itemize}
\item[I.] Fix a term order $\prec$ on the monomials of $R$ and let $LT_{\prec}(f)$ be the $\prec$-largest monomial of $f$, i.e., the \emph{leading term} of $f$. Compute the
\emph{initial ideal}
\[J:={\mathrm{init}}_{\prec}I=\langle {LT}_{\prec}(f): f\in I\rangle\]
by determining a Gr\"obner basis of $I$ with respect to $\prec$. 
\item[II.] If $J= \sqrt J$ is radical ($J$ is a squarefree monomial ideal), set $\widetilde{J}:=J$. Otherwise, let $\widetilde{J}$ be the (unique) \emph{polarization} of $J$, using a minimal generating set of $J$.
$\widetilde{J}$ is a squarefree monomial inside an enlarged polynomial ring $\widetilde R={\mathbb C}[{\widetilde x}_1,\ldots,
{\widetilde x}_M]$. 
\item[III.] Compute the prime decomposition \[\widetilde{J}=\bigcap_{j\in \Gamma} P_j.\] 
``Draw'' a \emph{hieroglyph} ${\sf H}_j$ for each $j\in \Gamma$. This is an array indexed by the variables ${\widetilde x}_i$ with $+$ in position 
${\widetilde x}_i$ if that variable
is a generator of $P_j$. The \emph{tablet} 
${\mathcal T}={\mathcal T}(V(I),\prec)$ is the set of hieroglyphs with the minimum number of $+$'s.
\end{itemize}

\begin{theorem}\label{thm:main}
$\deg V(I)=\#{\mathcal T}(V(I),\prec)$.
\end{theorem}

Call the algorithm for producing
$\#{\mathcal T}(V(I),\prec~)$ a
\emph{combinatorial commutative algebra rule}.
Step~II enables the universal application of the \emph{Stanley-Reisner correspondence}. Thus, $\#{\mathcal T}$ counts maximum-dimensional facets of the Stanley-Reisner complex for the squarefree monomial ideal~$\widetilde{J}$.

The motivation for Theorem~\ref{thm:main} stems from the influential work of Knutson--Miller on the Gr\"obner geometry of Schubert polynomials \cite{Knutson.Miller} and its follow-up research, see, e.g., \cite{KMS, Knutson, KMY, WY:Grob, Knutson:frob, LiLiYong, Weigandt.Yong,
Weigandt, Pawlowski, HPW, Klein.Weigandt} and references therein. In many instances, non-radical initial ideals have been viewed as obstacles, and one resorts to weighted counts when they cannot be circumvented. However, in combinatorial commutative algebra it is well-established that the polarization of monomial ideals reduces the problem to the radical case while preserving the degree. Based on the exploration of examples, it seems that integrating polarization into the framework of ideas presented in \cite{Knutson.Miller} and its related works is a crucial and necessary ingredient. This addition becomes particularly relevant when squarefree limits cannot exist, such as in cases where $V(I)$ itself is non-reduced.

\subsection{Warmup} We start with  two toy examples to illustrate the character of Theorem~\ref{thm:main}.

\begin{example}[Rank $\leq 1$ matrices]\label{exa:firstexa}
Let $I$ be the ideal generated by $2\times 2$ minors of a generic $3\times 3$ matrix
$Z=\left(\begin{smallmatrix} x_{11} & x_{12} & x_{13} \\ x_{21} & x_{22} & x_{23} \\ x_{31} & x_{32} & x_{33} \end{smallmatrix}\right)$. Let $\prec$ be the lexicographic order obtained by reading the rows of $Z$ in English
order. Here the initial ideal is squarefree:
\[{\widetilde J}=J=\langle x_{22}x_{33}, x_{21}x_{33}, x_{21}x_{32}, x_{12}x_{33}, x_{12}x_{23}, x_{11}x_{33}, x_{11}x_{32}, x_{11}x_{23}, x_{11}x_{22}\rangle\]
The primary decomposition is equidimensional:
\begin{align*} 
{\widetilde J}= & \langle x_{22}, x_{21}, x_{12}, x_{11}\rangle \cap
\langle x_{33}, x_{21}, x_{12}, x_{11}\rangle \cap 
\langle x_{33}, x_{32}, x_{12}, x_{11}\rangle \\ 
\ & \cap \langle x_{33}, x_{23}, x_{21}, x_{11}\rangle 
\cap \langle x_{33}, x_{32}, x_{23}, x_{11}\rangle \cap \langle x_{33}, x_{32}, x_{23}, x_{22}\rangle
\end{align*}
For each component, if $x_{ij}$ is a generator, we plot a $+$ in row $i$ and column $j$ in the associated hieroglyph. The resulting tablet is:
\[
\boxed{\boxed{\begin{matrix}
+ & + & \cdot\\
+ & + & \cdot\\
\cdot & \cdot & \cdot\end{matrix}} \ \  
\boxed{\begin{matrix}
+ & + & \cdot\\
+ & \cdot & \cdot\\
\cdot & \cdot & +\end{matrix}} \ \  
\boxed{\begin{matrix}
+ & + & \cdot\\
\cdot & \cdot & \cdot\\
\cdot & + & + \end{matrix}}
\ \ 
\boxed{\begin{matrix}
+ & \cdot & \cdot\\
+ & \cdot & +\\
\cdot & \cdot & + \end{matrix}}
\ \   
\boxed{\begin{matrix}
+ & \cdot & \cdot\\
\cdot & \cdot & +\\
\cdot & + & + \end{matrix}}
\ \   
\boxed{\begin{matrix}
\cdot & \cdot & \cdot\\
\cdot & + & +\\
\cdot & + & + \end{matrix}}}.
\]
Therefore, 
$\deg(\text{rank $\leq 1$, order $3$ matrices})=6$. The degree equals the number of semistandard tableaux of shape $(2,2,0)$ using entries from
$\{1,2,3\}$, namely
\[\tableau{1 & 1\\ 2 & 2} \ \ \ \  \tableau{1 & 1 \\ 2 & 3} \ \ \ \ \tableau{ 1 & 1 \\ 3 & 3}  \ \ \ \ 
\tableau{1 & 2 \\ 2 & 3} \ \ \ \ \tableau{1 & 2 \\ 3 & 3} \ \ \ \ \tableau{2 & 2 \\ 3 & 3}.\footnote{The count is also the dimension of the irreducible representation $V_{(2,2,0)}$ of $GL_3({\mathbb C})$.} 
\] 
It is an exercise to give a natural
bijection between the tablet and the tableaux.\footnote{The general solution can be found in \cite{KMY}.} In this case, Theorem~\ref{thm:main} produces a tablet that can be understood combinatorially. The interpretation depends on reading the tablet \emph{as a whole}, rather than the individual hieroglyphs.  \qed
\end{example}

\begin{example}[Rank $\leq 1$ symmetric matrices]
Let $I$ be generated by $2\times 2$ minors of a generic $3\times 3$ symmetric matrix
$Z=\left(\begin{smallmatrix} x_{11} & x_{12} & x_{13} \\ x_{12} & x_{22} & x_{23} \\ x_{13} & x_{23} & x_{33} \end{smallmatrix}\right)$.
Let $\prec$ be reverse lexicographic order. 
In this case the initial ideal is not squarefree:
\[J=\langle x_{23}^2, x_{13}x_{23}, x_{13}x_{22}, x_{13}^2, x_{12}x_{13}, x_{12}^2\rangle.\footnote{If one instead uses lexicographic
order, the initial ideal is squarefree; this emphasizes the role of choice.}\] 
Polarization of the initial ideal is needed. Each $x_{ij}^2$ that appears in $J$ becomes $y_{ij}z_{ij}$
in a new polynomial ring where $x_{ij}$ is replaced with $y_{ij}$ and $z_{ij}$.  Thus,
\[{\widetilde R}={\mathbb C}[x_{11},y_{12},z_{12},y_{13},z_{13},x_{22},y_{23},z_{23},x_{33}], \widetilde{J} = \langle y_{23}z_{23}, y_{13}y_{23}, y_{13}x_{22}, y_{	13}z_{13}, y_{12}y_{13}, y_{12}z_{12}\rangle.\]
$\widetilde{J}$ is not equidimensional:
\[\widetilde{J}= \langle y_{23}, y_{13}, y_{12}\rangle \cap 
\langle z_{23}, y_{13}, y_{12}\rangle \cap 
\langle y_{23}, y_{13}, z_{12}\rangle
\cap 
\langle z_{23}, y_{13}, z_{12}\rangle \cap \langle y_{23}, x_{22}, 
z_{13}, y_{12}\rangle.\]
Only the first four components contribute to the tablet. ${\sf H}_j$
has a $+$ in position $(i,j)$ if $x_{ij}$ or $y_{ij}$ is a generator of $P_j$; otherwise use $\oplus$.\footnote{We think of the markings in a hieroglyph as carvings with different depths.} The tablet is:
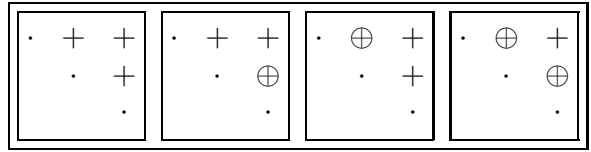
\[\boxed{\boxed{\begin{matrix} \cdot & + & + \\ \ & \cdot & + \\ \ & \ & \cdot\end{matrix}} \ \   
\boxed{\begin{matrix} \cdot & + & + \\ \ & \cdot & \oplus \\ \ & \ & \cdot\end{matrix}} \ \  
\boxed{\begin{matrix} \cdot & \oplus & + \\ \ & \cdot & + \\ \ & \ & \cdot\end{matrix}} \ \  
\boxed{\begin{matrix} \cdot & \oplus & + \\ \ & \cdot & \oplus \\ \ & \ & \cdot\end{matrix}}}
\]
Hence, 
$\deg(\text{rank $\leq 1$, order $3$ symmetric matrices})=4$. What is the interpretation?\qed
\end{example}

In these examples,
the natural matrix coordinates for the problem make the hieroglyphs graphical and combinatorially suggestive. The same holds for the main examples.

\subsection{Summary of the rest of this study}
Section~\ref{sec:2} provides a comprehensive review of relevant concepts from combinatorial commutative algebra and presents the proof of Theorem~\ref{thm:main}. In fact, a generalization to the multigraded case (Theorem~\ref{thm:strongermain}) is proved.

We then proceed to discuss three illustrative examples, recognizing that there remains a vast landscape to be explored with the algorithm.
 
Section~\ref{sec:MSV} examines \emph{matrix Schubert varieties}.
Theorem~\ref{thm:main} implies the
existence of a bijection between tablets for \emph{any} term order. Notably, for ``diagonal'' term orders this deviates from the approach of \cite{HPW,Klein.Weigandt} 
where multiplicities are used. New conjectures and problems arise as a result.

Section~\ref{sec:Commuting} takes a soujourn into the topic of the  \emph{commuting variety}. Knutson--Zinn-Justin have found an elegant weighted formula for the degree.
While this work does not offer a complete general formula like theirs, the theory reveals the existence of \emph{unweighted} rules. One sees explicit demonstrations of this for $n=2$ and $n=3$. From examining the $n=3$ tablet,  essential distinctions from the objects in their formula are
observed.

Section~\ref{sec:HS} considers \emph{Hilbert-Samuel multiplicities}. By virtue of Theorem~\ref{thm:main}, one surmises the existence of a specific form of rule, extending  \cite[Theorem~6.1]{LiLiYong} for Schubert varieties.

Theorem~\ref{thm:main} not only provides a systematic framework for obtaining concrete and visual sets with the desired enumerations, but also suggests that the solution to combinatorial challenges lies in deciphering the hieroglyphic tablets.

\section{Proof of Theorem~\ref{thm:main}}\label{sec:2}
\subsection{Preliminaries}
We review Hilbert series, $K$-polynomials, and multidegrees, as well as their connection to minimal free resolutions. Our primary reference is \cite{Miller.Sturmfels}. Let $R$ be the polynomial ring $\complexes[x_1,\dots, x_N]$. We give $R$ the standard grading (each $x_i$ has degree 1) and more generally a positive multigrading by assigning each $x_i$ a \textit{weight vector} $\mathbf{w}_i$. If $I$ is a (multigraded) homogeneous ideal in $R$, then $R/I$ decomposes as a direct sum of vector spaces $(R/I)_k$ (or $(R/I)_{\mathbf{a}}$) spanned by the degree-$k$ (or multidegree $\mathbf{a}$) polynomials in $R/I$. These vector spaces are all finite-dimensional because the (multi)grading on $R/I$ is positive \cite[Theorem 8.6]{Miller.Sturmfels}.

\begin{definition}
    Let $I\subseteq R$ be a homogeneous ideal. The \textit{Hilbert function} 
    $HF_{R/I}:\naturals\to\naturals$
    is the function defined by $HF_{R/I}(k) =  \dim_\complexes(R/I)_k$. The \textit{Hilbert series} of $R/I$ is the formal generating series for the Hilbert function:
    $$HS_{R/I}(t) = \sum_{k=0}^\infty HF_{R/I}(k)t^k.$$
\end{definition}

$HS_{R/I}(t)$ can be rewritten as $\frac{K_{R/I}(t)}{(1-t)^N}$, where the numerator is a polynomial in $t$ called the \textit{K-polynomial} of $R/I$. The \textit{degree} of $R/I$ is the coefficient on the lowest-degree term of $K_{R/I}(1-t)$. Analogously, the \textit{multigraded Hilbert series} is the formal generating series 
$$HS_{R/I}(\mathbf{t}) = \sum_{\mathbf{a}\in\integers^N}\dim_\complexes(R/I)_{\mathbf{a}}\ \mathbf{t}^{\mathbf{a}} = \frac{K_{R/I}(\mathbf{t})}{\prod_{i=1}^N(1-\mathbf{t}^{\mathbf{w}_i})}.$$
The numerator of this last expression is the \textit{multigraded K-polynomial} of $R/I$, and the \textit{multidegree} ${\mathrm {mdeg}}_{R/I}(\mathbf{t})$ of $R/I$ is the sum of the terms of lowest total degree in $K_{R/I}(\mathbf{1}-\mathbf{t})$. The multigraded $K$-polynomial is a Laurent polynomial \cite[Theorem 8.20]{Miller.Sturmfels}.

The multigraded $K$-polynomial of $R/I$ can also be computed using finite free resolutions. For a vector $\mathbf{a}$ in $\integers^N$, let $R(-\mathbf{a})$ be the free multigraded $R$-module in which $x_i$ has weight $\mathbf{a}+\mathbf{w}_i$. A \textit{(finite multigraded) free resolution} $F_\bullet$ of $R/I$ is a sequence of free $R$-modules connected by multidegree-0 maps $\partial_i$ in the following form:
$$0\to \bigoplus_{\mathbf{a}\in\integers^N} S(-\mathbf{a})^{\oplus b_{k,\mathbf{a}}}\xrightarrow{\partial_k}\bigoplus_{\mathbf{a}\in\integers^N} S(-\mathbf{a})^{\oplus b_{k-1,\mathbf{a}}}\xrightarrow{\partial_{k-1}}\cdots\xrightarrow{\partial_1}\bigoplus_{\mathbf{a}\in\integers^N} S(-\mathbf{a})^{\oplus b_{0,\mathbf{a}}}\to R/I\to 0.$$
We require $F_\bullet$ be \textit{exact}, meaning the image of each map is the kernel of the next. Finite free resolutions of this form always exist for $R/I$ \cite[Proposition 8.18]{Miller.Sturmfels}. A free resolution is \textit{minimal} if it simultaneously minimizes the values of all $b_{i, \mathbf{a}}$. Minimal free resolutions are unique up to isomorphism. If $F_\bullet$ is the minimal free resolution for $R/I$ then the associated values $b_{i, \mathbf{a}}$ are called the \textit{(multigraded) Betti numbers} of $R/I$ and denoted $\beta_{i,\mathbf{a}}$. The multigraded $K$-polynomial of $R/I$ can be computed directly from these Betti numbers:

\begin{proposition}[{\cite[Proposition 8.23]{Miller.Sturmfels}}]
    If $I\subseteq R$ is a multigraded homogeneous ideal, then
    $$K_{R/I}(\mathbf{t}) = \sum_{\substack{\mathbf{a}\in\integers^N \\ i\geq 0}}(-1)^i\beta_{i,\mathbf{a}}\mathbf{t}^{\mathbf{a}}.$$
\end{proposition}

In fact, the $\beta_{i,\mathbf{a}}$ in the above formula can be replaced by the values $b_{i,\mathbf{a}}$ appearing in any finite free resolution of $R/I$ \cite[Theorem 8.34]{Miller.Sturmfels}. We will use this fact below to show that polarization preserves $K$-polynomials.

\subsection{Gr\"obner degeneration}
The first step of our algorithm involves computing the \textit{initial ideal} $\init_\prec(I) = \langle LT_\prec(f):f\in I\rangle$ for a chosen term order $\prec$. 

\begin{definition}
    A \textit{Gr\"obner basis} for an ideal $I\subset R$ and term order $\prec$ is a finite set $\{g_1,\dots,g_s\}$ such that $\init_\prec(I) = \langle LT_\prec(g_1),\dots, LT_\prec(g_s)\rangle$. 
\end{definition}

Given any generating set $G$ for an ideal $I\subseteq \complexes[x_1,\dots,x_N]$, \textit{Buchberger's algorithm} enlarges $G$ to a Gr\"obner basis for $I$ with respect to $\prec$ \cite[Theorem 2.7.2]{CLO}. 
Futhermore, dividing a polynomial $f$ by the elements of any Gr\"obner basis $G$ for $I$ yields a unique remainder $r$ such that $f-r\in I$ and no term of $r$ is divisible by the lead term of any $g\in G$ \cite[Proposition 2.6.1]{CLO}. This remainder is  the \textit{normal form} of $f$ with respect to $I$ (and $\prec$).

\begin{lemma}\label{lemma:monomialbasis}
    Let $I\subseteq R = \complexes[x_1,\dots,x_N]$ be an ideal. Then $R/I$ has basis given by monomials not in $\init_\prec(I)$ when viewed as a vector space over $\complexes$.
\end{lemma}
\begin{proof}
    By definition, $f$ and its normal form $r$ with respect to $I$ represent the same equivalence class in $R/I$. No term of $r$ lies in $\init_\prec(I)$, so the monomials not in $\init_\prec(I)$ span $R/I$. Since a polynomial lies in a monomial ideal if and only if each of its terms do \cite[Lemma 2.4.3]{CLO}, the monomials not in $\init_\prec(I)$ are also linearly independent in $R/I$.
\end{proof}

\begin{corollary}\label{cor:grobnerpreserves}
    The Hilbert functions (and therefore $K$-polynomials and multidegrees) of $R/I$ and $R/\init_\prec(I)$ are equal.
\end{corollary}
\begin{proof}
    By Lemma~\ref{lemma:monomialbasis} the vector spaces whose dimensions are recorded by the Hilbert function have bases only depending on $\init_\prec(I)$, not $I$ itself.
\end{proof}

We note also that $\init_\prec(I)$ has a unique minimal generating set of monomials. This generating set can be easily computed from a Gr\"obner basis $G$ for $I$: it is the set of lead terms of $g\in G$ not divisible by the lead terms of any other $g'\in G$.

\subsection{Polarization}
When attempting to compute the Hilbert function (or $K$-polynomial or multidegree) of an ideal $I$, one can reduce to the case where $I$ is monomial by computing a Gr\"obner basis with respect to some term order $\prec$ and studying $\init_\prec(I)$. This approach is especially fruitful when $\init_\prec(I)$ is \textit{squarefree}, meaning its minimal monomial generators are all squarefree (this is equivalent to $\init_\prec(I)$ being a radical ideal). The \textit{polarization} operation converts any monomial ideal into a squarefree monomial ideal.

\begin{definition}
    Let $I = \langle g_1,\dots, g_s\rangle$ be a monomial ideal in $R = \complexes[x_1,\dots, x_N]$ with minimal monomial generators 
    \[g_j = \prod_{i=1}^N x_i^{a_{ij}}.\] 
    For $1\leq i\leq N$, let $m_i = \max_j\{a_{ij}\}$ be the largest power of $x_i$ appearing in some $g_j$. Let 
    \[\widetilde{R} = \complexes[x_{11},\dots,x_{1m_1},x_{21},\dots,x_{2m_2},\dots,x_{Nm_N}].\] 
    The \textit{polarization} of $I$ is the squarefree monomial ideal $\widetilde{I}\subseteq\widetilde{R}$ generated by $\{\widetilde{g_j}\}_{j=1}^s$, where
    $$\widetilde{g}_j = \prod_{i=1}^N\prod_{k=1}^{a_{ij}}x_{ik}.$$
\end{definition}

\begin{example}
    Let $I\subseteq\complexes[x_1, x_2]$ be the monomial ideal generated by $x_1^3x_2$ and $x_2^2$. Then $\widetilde{R} = \complexes[x_{11}, x_{12}, x_{13}, x_{21}, x_{22}]$ and $\widetilde{I} = \langle x_{11}x_{12}x_{13}x_{21}, x_{21}x_{22}\rangle$.\qed
\end{example}

In general, $R/I$ and $\widetilde{R}/\widetilde{I}$ do not have the same Hilbert function. However, we will show that they have the same $K$-polynomial and thus the same multidegree. To prove this it is easier to work with \textit{partial polarization}, which only adds one variable at a time. This argument, in the standard graded case, appears in Section 1.6 of \cite{Herzog}. We have not found an explicit reference for the multigraded case and therefore provide proofs for convenience.

\begin{definition}
    Let $I = \langle g_1,\dots, g_s\rangle$ be as above and let $i$ be an index such that $m_i\geq 2$. The \textit{partial polarization} of $I$ with respect to $x_i$ is the ideal $I'$ in $R' = \complexes[x_1,\dots,x_N, y]$ obtained by replacing each generator $g_j$ such that $a_{ij}\geq 2$ with $g'_j = \frac{g_j}{x_i}y$.
\end{definition}

\begin{example}
    With $I  = \langle x_1^3x_2, x_2^2\rangle$, the partial polarization with respect to $x_1$ is $I' = \langle x_1^2yx_2, x_2^2\rangle$ in $\complexes[x_1, x_2, y]$. The partial polarization with respect to $x_2$ is $I'' = \langle x_1^3x_2, x_2y\rangle$.
\end{example}

The polarization $\widetilde{I}$ of $I$ can clearly be computed by iteratively taking partial polarizations, so it suffices to show that partial polarization preserves the $K$-polynomial. 

\begin{lemma}[{\cite[Lemma 1.6.1]{Herzog}}]\label{lemma:zerodivisor}
    If $I\subseteq R$ is a monomial ideal and $I'\subseteq R'$ is the partial polarization of $I$ with respect to $x_i$, then $y-x_i$ is not a zero-divisor in $R'/I'$.
\end{lemma}
\begin{proof}
    For simplicity of notation let $x = x_i$ be the variable we perform partial polarization with respect to. Suppose $y-x$ is a zero-divisor in $R'/I'$. Then there exists an element $f\in R'\setminus I'$ such that $f(y-x)$ lies in $I'$. Since $I'$ is a monomial ideal, we may take $f$ to be a monomial such that $fy$ and $fx$ both lie in $I'$. Let $g$ and $h$ be minimal monomial generators of $I'$ dividing $fy$ and $fx$ respectively. Write $f = x^af'$, $g = x^bg'$ and $h = x^ch'$, where $f'$, $g'$ and $h'$ are not divisible by $x$.
    
    Since $f\notin I'$, $g$ cannot divide $f$, so $g$ contains $y$. Since $g$ divides $fy$, $a\geq b$. Similarly, $h$ divides $fx$ but not $f$, so $c=a+1$. Thus $c > b$. Since $g$ and $h$ are minimal monomial generators of $I'$ and $g$ contains $y$ it follows from the partial polarization construction that $h$ also contains $y$. Since $h$ divides $fx$, this implies that $f$ contains $y$.

    Write $gw = fy$ for some monomial $w$. Then $w$ cannot contain $y$ (otherwise dividing both sides by $y$ expresses $f$ as an element of $I'$). Thus $g$ is a minimal monomial generator of $I'$ containing a factor of $y^2$, contradicting the definition of partial polarization. 
\end{proof}

Generate $\widetilde{I}$ from $I$ via a sequence of partial polarizations and let $L$ be the ideal generated by the sequence of $y-x_i$ terms corresponding to each step, which form a regular sequence by Lemma~\ref{lemma:zerodivisor}. Then 
$R/I\cong \widetilde{R}/(\widetilde{I}+L)$ 
as graded $\complexes$-algebras. If $R = \complexes[x_1,\dots,x_N]$ is given a positive multigrading where $x_i$ has weight $\mathbf{w}_i$, then in the partial polarization $R' = \complexes[x_1,\dots,x_N, y]$ with respect to $x_i$ we assign $y$ the weight $\mathbf{w}_i$. Then $y-x_i$ is a multigraded homogeneous polynomial in $R'$ and it follows immediately that $R/I\cong\widetilde{R}/(\widetilde{I}+L)$ as multigraded $\complexes$-algebras. From this we can show that polarization preserves the multigraded $K$-polynomial and therefore the multidegree.

\begin{theorem}\label{thm:polarizationpreserves}
    Let $I\subseteq R$ be a monomial ideal and let $\widetilde{I}\subseteq\widetilde{R}$ be its polarization. Then $K_{R/I}(\mathbf{t}) = K_{\widetilde{R}/\widetilde{I}}(\mathbf{t})$. In particular, $\mathrm{mdeg}_{R/I}(\mathbf{t}) = \mathrm{mdeg}_{\widetilde{R}/\widetilde{I}}(\mathbf{t})$.
\end{theorem}
\begin{proof}
    It suffices to prove the result for the partial polarization $I'\subseteq R'$ with respect to $x_i$. Let $F_\bullet$ be a finite free resolution of $R'/I'$ as an $R'$ module. We claim that the tensor product 
    \[G_\bullet = F_\bullet\otimes_{R'}R'/\langle y-x_i\rangle\] 
    is a free resolution of $R/I$. It follows from basic properties of tensor products that \[R'\otimes_{R'}R'/\langle y-x_i\rangle\cong R \text{ \ and \ $R'/I'\otimes_{R'}R'/\langle y-x_i\rangle\cong R'/(I'+\langle y-x_i\rangle)$.}\] 
    By Lemma~\ref{lemma:zerodivisor} we know that $R'/(I'+\langle y-x_i\rangle) \cong R/I$. Thus $G_\bullet$ is a sequence of free $R$-modules ending in $R/I$. To prove $G_\bullet$ is a free resolution it remains to show that this sequence is exact. This is equivalent to verifying that 
    \[{\mathrm {Tor}}_i^{R'}(R'/I', R'/\langle y-x_i\rangle) = 0\] for all positive $i$. By the symmetry of ${\mathrm{Tor}}$ (\cite[Exercise 1.12]{Miller.Sturmfels}), this can be computed from the tensor product of $R'/I'$ with an $R'$-free resolution of $R'/\langle y-x_i\rangle$ by taking homology. We use the free resolution
    $$0\to R'\xrightarrow{y-x_i} R'\to R'/\langle y-x_i\rangle\to 0,$$
    and apply the functor $(R'/I')\otimes_{R'}-$ to obtain the complex
    $$0\to R'/I'\xrightarrow{y-x_i} R'/I'\to R/I\to 0.$$
    This complex is still exact since $(y-x_i)$ is not a zero-divisor in $R'/I'$, so the desired $\mathrm{Tor}$ groups vanish and $G_\bullet$ is an $R$-free resolution for $R/I$ as desired. Since the $K$-polynomial can be computed from any finite free resolution for a module, $K_{R'/I'}$ and $K_{R/I}$ are the same. Thus the $K$-polynomials of $R/I$ and its polarization $\widetilde{R}/\widetilde{I}$ are the same.
\end{proof}

\subsection{Stanley-Reisner theory}
Gr\"obner degeneration and polarization allow us to relate any ideal $I\subseteq \complexes[x_1,\dots, x_N]$ to a squarefree monomial ideal with the same $K$-polynomial (and therefore the same multidegree). Any squarefree monomial ideal $I$ is radical and can therefore be expressed uniquely as the intersection of all minimal prime ideals $P$ containing it. These primes are monomial ideals, and any prime monomial ideal is of the form $P_A = \langle x_i : i\in A\rangle$ for some subset $A$ of $[N] = \{1,\dots, N\}$. This helps us compute multidegrees of squarefree monomial ideals via the \textit{Stanley-Reisner correspondence}.

\begin{definition}
    A \textit{simplicial complex} on $[N]$ is a collection of subsets $\Delta\subseteq 2^{[N]}$ that is closed under subsets (so if $\sigma\in\Delta$ then any subset of $\sigma$ lies in $\Delta$).
\end{definition}

The elements of a simplicial complex $\Delta$ are called its \textit{faces}, and the faces of $\Delta$ that are maximal under inclusion are called \textit{facets}. The cardinality of a face is its \textit{dimension}.

\begin{definition}
    The \textit{Stanley-Reisner complex} of a squarefree monomial ideal $I\subseteq R$ is the simplicial complex $\Delta_I$ on $[N]$ whose faces are the sets $\{i_1,\dots ,i_k\}\subset [N]$ such that $x_{i_1}\dots x_{i_k}$ does not lie in $I$.
\end{definition}

\begin{theorem}[{\cite[Theorem 1.7]{Miller.Sturmfels}}]\label{thm:stanleyreisner}
    There is a bijection $I\leftrightarrow \Delta_I$ between squarefree monomial ideals in $\complexes[x_1,\dots, x_N]$ and simplicial complexes on $[N]$. Furthermore, 
    \[I_\Delta = \bigcap_{\sigma}P_{\sigma^c},\] 
    where the intersection runs over facets $\sigma$ of $\Delta$ and $\sigma^c$ is the complement of $\sigma$ in $[N]$.
\end{theorem}

\begin{theorem}[{\cite[Theorem 1.13]{Miller.Sturmfels}}]\label{thm:Kformula}
    Let $I\subseteq R$ be a squarefree monomial ideal. Then the (multigraded) $K$-polynomial of $R/I$ is
    $$K_{R/I}(\mathbf{t}) = \sum_{\sigma\in\Delta_I}\left(\prod_{i\in \sigma}\mathbf{t}^{\mathbf{w}_i}\prod_{j\notin \sigma}(1-\mathbf{t}^{\mathbf{w}_j})\right).$$
\end{theorem}

We can now prove a stronger version of Theorem~\ref{thm:main}, from which Theorem~\ref{thm:main} itself follows immediately.

\begin{theorem}\label{thm:strongermain}
    Let $R = \complexes[x_1,\dots,x_N]$ be a positively multigraded ring where each $x_i$ has the same total degree (the sum of the components of each ${\bf w}_i$ is the same). Let $I\subseteq R$ be a multigraded homogeneous ideal, let $\prec$ be a term order, and let ${\mathcal T}={\mathcal T}(V(I),\prec)$ be the associated tablet. Then 
    $${\mathrm{mdeg}}_{R/I}(\mathbf{t}) = \sum_{{\sf{H}} \in{\mathcal T}} \prod_{\tilde{x}_i\in {\sf{H}}} \mathbf{t}^{\mathbf{w}_i}.$$
\end{theorem}
\begin{proof}
    Let $I\subseteq R$ be a multigraded homogeneous ideal and let $\prec$ be a term order. Then $J = \init_\prec(I)$ is a monomial ideal such that $K_{R/I}(\mathbf{t}) = K_{R/J}(\mathbf{t})$ by Corollary~\ref{cor:grobnerpreserves}. The polarization $\widetilde{J}\subseteq\widetilde{R}$ of $J$ is then a squarefree monomial ideal such that 
    $K_{\widetilde{R}/\widetilde{J}}(\mathbf{t}) = K_{R/I}(\mathbf{t})$ 
    by Theorem~\ref{thm:polarizationpreserves}. In particular the multidegrees of $R/I$ and $\widetilde{R}/\widetilde{J}$ agree.
    
    For simplicity of notation we now assume that $I\subseteq R$ is a squarefree monomial ideal. From Theorem~\ref{thm:Kformula},
    $$K_{R/I}(\mathbf{1}-\mathbf{t}) = \sum_{\sigma\in\Delta_I}\prod_{i\in\sigma}(\mathbf{1}-\mathbf{t})^{\mathbf{w}_i}\prod_{j\notin\sigma}(1-(\mathbf{1}-\mathbf{t})^{\mathbf{w}_j}).$$

    Let $m$ be the maximum dimension of a face in $\Delta_I$. Then it follows easily from definitions, together with our ``total degree'' hypothesis, that the multidegree of $R/I$ is
    $${\mathrm {mdeg}}_{R/I}(\mathbf{t}) = \sum_{\substack{\sigma\in\Delta_I \\ |\sigma| = m}}\prod_{j\notin\sigma}\mathbf{t}^{\mathbf{w}_j},$$ where the sum is over dimension-$m$ facets of $\Delta_I$. By the unique prime decomposition of $I$ given in Theorem~\ref{thm:stanleyreisner}, a term $\mathbf{t}^{\mathbf{w}}$ appears in the multidegree of $R/I$ if and only if $\mathbf{w} = \sum_{j=1}^m \mathbf{w}_{i_j}$ is the weight vector for the product of all variables in a maximum-dimensional component of the prime decomposition. The algorithm in the introduction records one hieroglyph for each such component, giving the desired formula.
\end{proof}
    
\begin{proof}[Proof of Theorem~\ref{thm:main}]
    Pick the multigrading giving each $x_i$ weight $1$ in Theorem~\ref{thm:strongermain}.
\end{proof}

\section{Matrix Schubert varieties} \label{sec:MSV}

\subsection{Definition} Let ${\sf Mat}_n$ be the space of $n\times n$ matrices with complex
entries. Identify $R:={\mathbb C}[{\sf Mat}_n]$ with ${\mathbb C}[z_{ij}:1\leq i,j\leq n]$ where $z_{ij}$ is the $(i,j)$-coordinate function. Fulton~\cite{Fulton:duke} defines generators for the \emph{Schubert determinantal ideal} $I_w$ as follows. Let 
$r_{ij}=\#\{1\leq a\leq i: w(a)\leq j\}$. Let $Z=(x_{ij})_{1\leq i,j\leq n}$ be the generic $n\times n$ matrix and $Z_{ij}$ be the northwest $i\times j$ submatrix of $Z$. Then
\[    I_w=\langle \text{$(r_{ij}+1)\times(r_{ij}+1)$ minors of $Z_{ij}$, $1\leq i,j\leq n$}\rangle\subseteq R.\]
Now, ${\mathfrak X}_w:=V(I_w)\subset {\mathbb C}^{n^2}$ is the \emph{matrix Schubert variety} for $w$.

\subsection{Antidiagonal term orders} Knutson-Miller's work \cite{Knutson.Miller} studies
\emph{antidiagonal term orders}
$\prec$ on the monomials of $R$. Such an order picks the
main southwest-northeast diagonal as the lead term of a
minor; an example is lexicographically ordering the variables by reading $Z$ along rows from right to left and top to bottom.

One may rephrase one of their  main results as follows:
\begin{theorem}[{\cite[Theorem~B]{Knutson.Miller}}]
\label{KMthm}
Let $\prec$ be any antidiagonal term order. The tablet ${\mathcal T}({\mathfrak X}_w,\prec)$ is in natural bijection with the pipe dreams (also known as rc-graphs) for $w$ \cite{BB, Fomin.Kirillov93}.
\end{theorem}
 
\subsection{Diagonal term orders} There is also interest in Gr\"obner degenerations of ${\mathfrak X}_w$ with respect to \emph{diagonal term orders} $\prec'$.
Such an order makes the lead term of a minor of $Z$ the main diagonal
(e.g., lexicographic order on the variables of $Z$ as read in English reading order, which we call \emph{lex-diagonal order} below). In \cite{KMY} one sees that the defining generators are a Gr\"obner basis of $I_w$ with respect to $\prec'$ if and only if $w$ is vexillary (avoids the permutation pattern $2143$), in which case the degenerations relate to semistandard Young tableaux.

\begin{example}\label{exa:214365}
Let $w=214365$ and pick $\prec'$ to be lex-diagonal order as above. Here $I=I_{w}=\langle\det Z_{11}, \det Z_{33}, \det Z_{55}\rangle$. The initial ideal is given by
\[J=\langle x_{13}x_{21}^2x_{32}x_{34}x_{43}x_{55}, \        
     x_{12}x_{23}x_{31}x_{34}x_{43}x_{55}, \
     x_{12}x_{21}x_{34}x_{43}x_{55}, x_{12}x_{21}x_{33}, \ x_{11} 
     \rangle.\]
Only $x_{21}$ appears with degree $\geq 2$ in the generators of $J$. To polarize $J$ we introduce one additional variable $y_{21}$ and define ${\widetilde R}={\mathbb C}[x_{11},\ldots, x_{66}, y_{21}]$. Then 
\[\widetilde{J}=\langle 
x_{13}\underline{x_{21}y_{21}}x_{32}x_{34}x_{43}x_{55}, \        
     x_{12}x_{23}x_{31}x_{34}x_{43}x_{55}, \
     x_{12}x_{21}x_{34}x_{43}x_{55}, x_{12}x_{21}x_{33}, \ x_{11} 
\rangle,
\]
where we have replaced $x_{21}^2$ with $x_{21}y_{21}$.

    Now $\deg {\mathfrak X}_w=15$ (there are fifteen pipe dreams for $w$). One checks that the prime decomposition
    of $\widetilde{J}$ is equidimensional with $15$ components:
 \begin{align*}
     \widetilde{J}= &  \langle x_{13}, x_{12}, x_{11}\rangle 
 \cap \langle x_{21}, x_{12}, x_{11}\rangle 
 \cap \langle y_{21}, x_{12}, x_{11}\rangle
 \cap \langle x_{32}, x_{12}, x_{11}\rangle
 \cap \langle x_{34}, x_{12}, x_{11}\rangle \\
\ & \cap \langle x_{43}, x_{12}, x_{11}\rangle 
\cap \langle x_{55}, x_{12}, x_{11}\rangle
\cap \langle x_{23}, y_{21}, x_{11}\rangle 
\cap \langle x_{31}, y_{21}, x_{11}\rangle
\cap \langle x_{34}, y_{21}, x_{11}\rangle \\
\ & \cap \langle x_{43}, y_{21}, x_{11}\rangle 
\cap \langle x_{55}, y_{21}, x_{11}\rangle
\cap  \langle x_{34}, x_{33}, x_{11}\rangle
\cap \langle x_{43}, x_{33}, x_{11}\rangle
\cap \langle x_{55}, x_{33}, x_{11}\rangle. 
 \end{align*}

    
    If we graph the components with the rule that $x_{21}\to +$ and $y_{21}\to \oplus$, there are two hieroglyphs with the same support in the $6\times 6$ grid, namely 
    \[\boxed{\begin{matrix}
    + & + & \cdot & \cdot & \cdot & \cdot \\
    + & \cdot & \cdot & \cdot & \cdot & \cdot \\
\cdot & \cdot & \cdot & \cdot & \cdot & \cdot \\
\cdot & \cdot & \cdot & \cdot & \cdot & \cdot \\
\cdot & \cdot & \cdot & \cdot & \cdot & \cdot \\
\cdot & \cdot & \cdot & \cdot & \cdot & \cdot \end{matrix}} \ \ \ \ 
\boxed{\begin{matrix}
    + & + & \cdot & \cdot & \cdot & \cdot \\
    \oplus & \cdot & \cdot & \cdot & \cdot & \cdot \\
\cdot & \cdot & \cdot & \cdot & \cdot & \cdot \\
\cdot & \cdot & \cdot & \cdot & \cdot & \cdot \\
\cdot & \cdot & \cdot & \cdot & \cdot & \cdot \\
\cdot & \cdot & \cdot & \cdot & \cdot & \cdot \end{matrix}}.
\]
All other hieroglyphs have distinct support.\qed
     \end{example}

The following conjecture holds for
$n\leq 7$ when $\prec'$ is the lex-diagonal order.

\begin{conjecture}\label{conj:equidim}
Under $\prec'$, the ideal $\widetilde{J}$ obtained by
polarizing ${\mathrm{init}}_{\prec'} I_w$ is equidimensional.
\end{conjecture}

Conjecture~\ref{conj:equidim} suggests that
the Stanley-Reisner complex $\Delta_{\widetilde{J}}$ associated to 
$\widetilde{J}$ is nice.

\begin{problem} In increasing order of strength: is the Stanley-Reisner simplicial complex $\Delta_{\widetilde{J}}$
Cohen-Macaulay, shellable, vertex-decomposible, and/or 
a subword complex~\cite{Knutson.Miller:subword}? 
\end{problem}

In the case of antidiagonal term orders, the answer to the final question (and thus all others) is ``yes'' \cite{Knutson.Miller, Knutson.Miller:subword}. The same is true for diagonal term orders under the
vexillary hypothesis \cite{KMY} (see also the related paper \cite{KMY2} on \emph{tableau complexes}).

Hamaker-Pechenik-Weigandt \cite[Conjecture~1.2]{HPW} give a beautiful conjecture that the irreducible components of $V({\mathrm{init}}_{\prec'}(I_w))$ \emph{counted with multiplicity} naturally correspond to the bumpless pipe dreams (BPDs) of $w$ (we refer to \cite{HPW} for definitions and references). For some diagonal term orders, this is proved by Klein-Weigandt \cite{Klein.Weigandt}. The next
conjecture asserts a multiplicity-free version
of these findings.

\begin{conjecture}\label{conj:BPD}
The tablet ${\mathcal T}({\mathfrak X}_w,\prec')$ is in natural bijection with the BPDs for $w$, meaning the support of each hieroglyph gives the positions for the ``blank tiles'' of a BPD.
\end{conjecture}

Theorem~\ref{thm:main} proves the existence of \emph{a} bijection (but not a ``natural'' one). We invite the reader to confirm Conjecture~\ref{conj:BPD} for $w=214365$ using the prime decomposition above. In particular, there are
exactly two BPDs with blank tiles in positions $(1,1),(1,2),(2,1)$ corresponding to the two hieroglyphs of Example~\ref{exa:214365}. We confess that Conjecture~\ref{conj:BPD} has only been really checked in examples for the
lex-diagonal term order.

\subsection{Other term orders}
For each of the (finitely many) other initial ideals for $I_w$ one obtains through Theorem~\ref{thm:main} alternatives to pipe dreams and bumpless pipe dreams. 

\begin{example}\label{exa:2143} Let $w=2143$. Here $I_w$ is generated by the northwest $1\times 1$
and $3\times 3$ minors of the generic matrix 
$Z=\left(\begin{smallmatrix} x_{11} & x_{12} & x_{13} & x_{14} \\ x_{21} & x_{22} & x_{23} &x_{24}  \\ x_{31} & x_{32} & x_{33} & x_{34}\\
x_{41} & x_{42} & x_{43} & x_{44}
\end{smallmatrix}\right)$. Let us start with the
lex-diagonal term order. In this case, the initial ideal is squarefree with 
${\widetilde J}=J=\langle x_{12},x_{11}\rangle \cap
\langle x_{21}, x_{11}\rangle \cap
\langle x_{33}, x_{11}\rangle$.
The tablet is therefore:
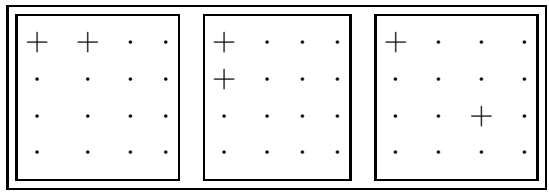
\[
\boxed{\boxed{\begin{matrix}
+ & + & \cdot & \cdot \\
\cdot & \cdot & \cdot & \cdot \\
\cdot & \cdot & \cdot & \cdot \\
\cdot & \cdot & \cdot & \cdot 
\end{matrix}} \ \ \ 
\boxed{\begin{matrix}
+ & \cdot & \cdot & \cdot \\
+ & \cdot & \cdot & \cdot \\
\cdot & \cdot & \cdot & \cdot \\
\cdot & \cdot & \cdot & \cdot 
\end{matrix}} \ \ \ 
\boxed{\begin{matrix}
+ & \cdot & \cdot & \cdot \\
\cdot & \cdot & \cdot & \cdot \\
\cdot & \cdot & + & \cdot \\
\cdot & \cdot & \cdot & \cdot 
\end{matrix}}}
\]
As the term order is lexicographic-diagonal,
Conjecture~\ref{conj:BPD} asserts that the hieroglyphs
in this tablet biject with BPDs for $w=2143$; this is true
here. 

For convenience, let us rename the lex-diagonal order the ``$1234$-order'' since we are using lex order on the variables read with the rows of $Z$ read in order $1,2,3,4$. Similarly, let us now discuss $1324$-order. Then the resulting tablet is
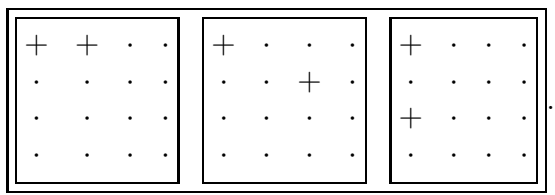
\[
\boxed{\boxed{\begin{matrix}
+ & + & \cdot & \cdot \\
\cdot & \cdot & \cdot & \cdot \\
\cdot & \cdot & \cdot & \cdot \\
\cdot & \cdot & \cdot & \cdot 
\end{matrix}} \ \ \ 
\boxed{\begin{matrix}
+ & \cdot & \cdot & \cdot \\
\cdot & \cdot & + & \cdot \\
\cdot & \cdot & \cdot & \cdot \\
\cdot & \cdot & \cdot & \cdot 
\end{matrix}} \ \ \ 
\boxed{\begin{matrix}
+ & \cdot & \cdot & \cdot \\
\cdot & \cdot & \cdot & \cdot \\
+ & \cdot & \cdot & \cdot \\
\cdot & \cdot & \cdot & \cdot 
\end{matrix}}}.
\]
 The tablets are the same for the orders 
 $1324, 3124, 3142, 3412$. However, a change in tablets occurs between $3412$ and $3421$, and
 $3421$ and $4321$ give the same tablet, which correspond to the pipe dreams:
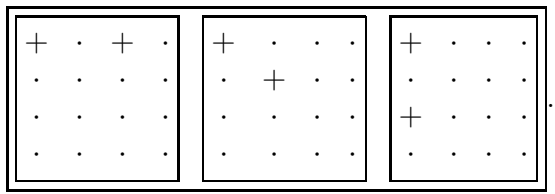
\[
\boxed{\boxed{\begin{matrix}
+ & \cdot & + & \cdot \\
\cdot & \cdot & \cdot & \cdot \\
\cdot & \cdot & \cdot & \cdot \\
\cdot & \cdot & \cdot & \cdot 
\end{matrix}} \ \ \
\boxed{\begin{matrix}
+ & \cdot & \cdot & \cdot \\
\cdot & + & \cdot & \cdot \\
\cdot & \cdot & \cdot & \cdot \\
\cdot & \cdot & \cdot & \cdot 
\end{matrix}} \ \ \ 
\boxed{\begin{matrix}
+ & \cdot & \cdot & \cdot \\
\cdot & \cdot & \cdot & \cdot \\
+ & \cdot & \cdot & \cdot \\
\cdot & \cdot & \cdot & \cdot 
\end{matrix}}}.
\]
This makes sense since $4321$ order is an antidiagonal term order. In this way, one interpolates between the BPDs and pipe dreams for $w=2143$. \qed
\end{example}

For any permutation $w$, Theorem~\ref{thm:main} implies the existence of a bijection between the tablets for different term orders.\footnote{Recent work of Knutson--Udell \url{https://youtu.be/C8xARISi_X0} gives an interpolation between BPDs and pipe dreams. It would be interesting to compare the analysis above with their results.} Moreover, by virtue of Theorem~\ref{thm:strongermain}, one knows that the various pipe dreams all compute the (double) Schubert polynomials (see \cite{Fulton, Manivel} and references therein). This \emph{automatically} implies the
existence of many new unweighted formulas for these important families of polynomials from Schubert calculus.

\section{The commuting variety}\label{sec:Commuting}

The commuting variety $C_n$ is the
reduced subscheme of ${\sf Mat}_n\oplus {\sf Mat}_n$
consisting of pairs of matrices $(A,B)$ such that 
$[A,B]=AB-BA=0$. The $n^2$ many ``obvious'' defining equations are not known to define a radical ideal
(although conjecturally they do); let $I_n$ be this ideal. Many values of $\deg C_n$ are now known \url{https://oeis.org/A029729}.
\[1, 3, 31, 1145, 154881, 77899563, 147226330175, 1053765855157617, \ldots.\]
Use of Gr\"obner degeneration to study $C_n$ was initiated by Knutson in \cite{Knutson}.
Work of Knutson--Zinn-Justin \url{https://pi.math.cornell.edu/~allenk/caac2022.pdf} gives an incredible formula for these degrees as a \emph{weighted} sum. 

For $n=2$, we have $A=\left(\begin{smallmatrix} a_{11} & a_{12}\\ a_{21} & a_{22}\end{smallmatrix}\right)$ and
$B=\left(\begin{smallmatrix} b_{11} & b_{12}\\ b_{21} & b_{22}\end{smallmatrix}\right)$. Now
\begin{multline}\nonumber
I=I_2(=\sqrt{I_2})=\langle
- a_{21}b_{12} + a_{12}b_{21}, - a_{12}b_{11} + a_{11}b_{12} - a_{22}b_{12} + a_{12}b_{22}, \\
     a_{21}b_{11} - a_{11}b_{21} + a_{22}b_{21} - a_{21}b_{22}, a_{21}b_{12} - a_{12}b_{21}
\rangle
\end{multline}
Using the term order that reads the variables in the $A$ matrix first followed by those in $B$, namely, 
$a_{11}\succ a_{12}\succ a_{21}\succ a_{22} \succ b_{11}\succ b_{12}
\succ b_{21} \succ b_{22}$,
one finds that the initial ideal is squarefree and has
prime decomposition encoded by the tablet
\[\boxed{\boxed{\begin{matrix} \cdot & + \\ + & \cdot\end{matrix}}
\boxed{\begin{matrix} \cdot & \cdot \\ \cdot & \cdot\end{matrix}} 
\ \  
\boxed{\begin{matrix} \cdot & \cdot \\ + & \cdot\end{matrix}}
\boxed{\begin{matrix} + & \cdot \\ \cdot & \cdot\end{matrix}}
\ \   
\boxed{\begin{matrix} \cdot & \cdot \\ \cdot & \cdot\end{matrix}}
\boxed{\begin{matrix} + & + \\ \cdot & \cdot\end{matrix}}}
\text{ \ using coordinates  \
$\boxed{\begin{matrix} a_{11} & a_{12} \\ a_{21} & a_{22} \end{matrix}}
\boxed{\begin{matrix} b_{11} & b_{12} \\ b_{21} & b_{22}\end{matrix}}$
}
\]

The $n=3$ case is interesting from the perspective of our framework as it needs polarization. Let $I=I_3(=\sqrt{I_3})$ be defined in the same way as $I_2$. Then one computes that 
\begin{align*}
J= & \langle a_{31}b_{13}, a_{21}b_{13}, a_{31}b_{12}, a_{21}b_{12}, a_{31}b_{11}, a_{21}b_{11}, a_{13}b_{11}, a_{12}b_{11}, a_{32}b_{13}b_{22}, a_{32}b_{13}b_{21}, a_{32}b_{12}b_{21}, \\
\ & a_{23}b_{12}b_{21}, a_{13}b_{12}b_{21},        a_{13}a_{32}b_{21}, a_{12}a_{32}b_{21}, a_{13}a_{31}b_{21},
        a_{12}a_{31}b_{21}, a_{11}a_{31}b_{21}, a_{13}a_{22}b_{21}, a_{11}a_{13}b_{12},\\
\ &       a_{12}a_{31}^2b_{23}, a_{12}a_{31}^2b_{22},
        a_{12}a_{23}a_{31}b_{22}, a_{13}a_{21}a_{31}b_{22}, a_{12}a_{13}a_{31}b_{22}, a_{13}^2a_{21}a_{32}b_{22}\rangle.
\end{align*}
The initial ideal is not squarefree as it contains $a_{12}a_{31}^2b_{23}$, $a_{12}a_{31}^2b_{22}$,
and $a_{13}^2a_{21}a_{32}b_{22}$. To polarize, we introduce 
$c_{13}$ and $c_{31}$ and replace the non-squarefree
generators accordingly.
The prime decomposition has $32$ components, $31$ of which achieve the maximal dimension $6$. Hence the degree of the commuting scheme is $31$, as desired. 

The hieroglyphs are encoded using the coordinates
\[
\boxed{\begin{matrix} 
a_{11} & a_{12} & a_{13} \\ 
a_{21} & a_{22} & a_{23}\\
a_{31} & a_{32} & a_{33}
\end{matrix}}
\boxed{\begin{matrix} b_{11} & b_{12} & b_{13} \\ 
b_{21} & b_{22} & b_{23}\\
b_{31} & b_{32} & b_{33}
\end{matrix}} \text{ \ and \ }
\boxed{\begin{matrix} 
a_{11} & a_{12} & c_{13} \\ 
a_{21} & a_{22} & a_{23}\\
c_{31} & a_{32} & a_{33}
\end{matrix}}
\boxed{\begin{matrix} b_{11} & b_{12} & b_{13} \\ 
b_{21} & b_{22} & b_{23}\\
b_{31} & b_{32} & b_{33}
\end{matrix}}
\]
where $a_{13},a_{31}\mapsto +$ and $c_{13},c_{31}\mapsto \oplus$. The tablet consists of these $31$ hieroglyphs:

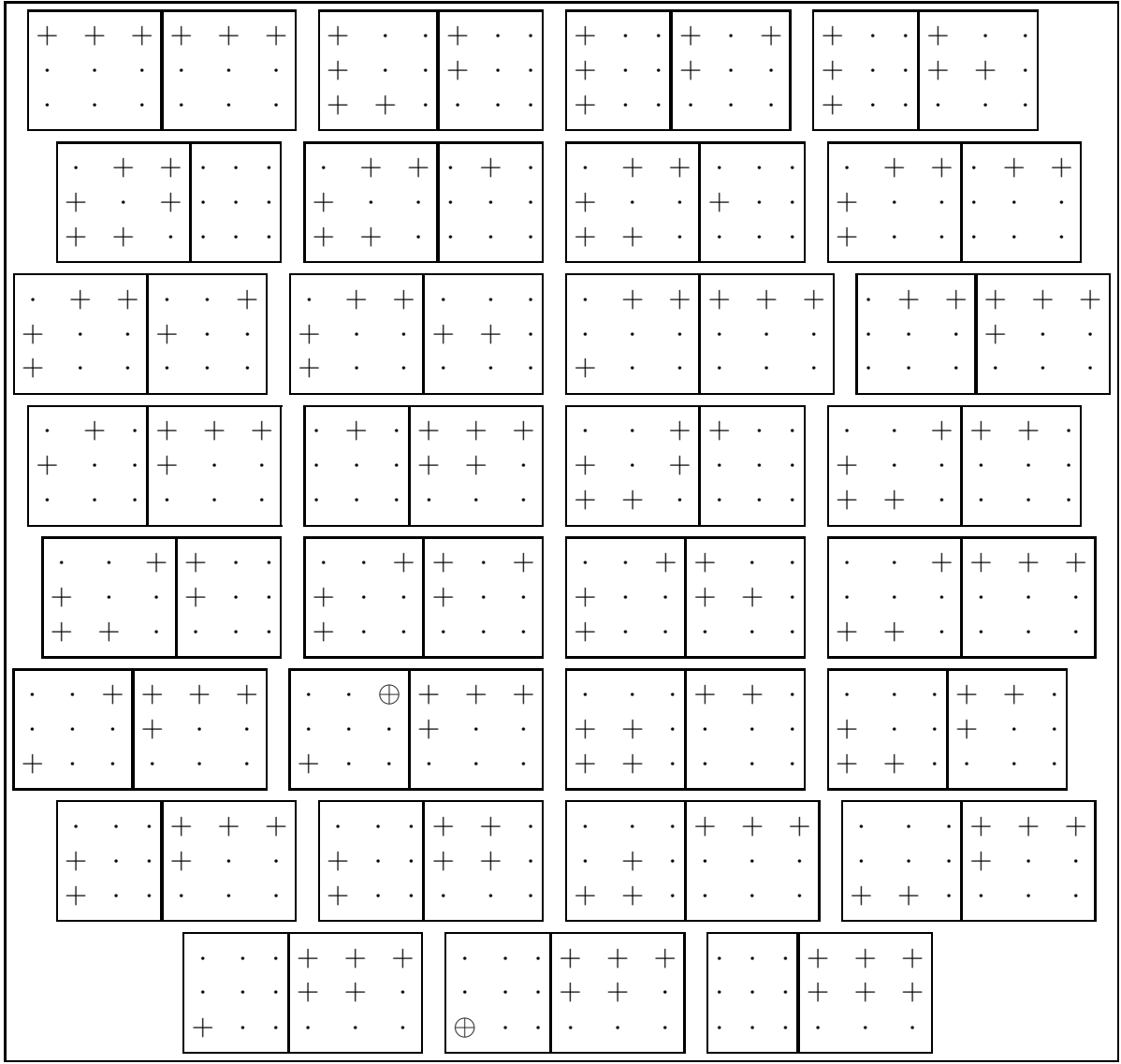
\begin{empheq}[box=\fbox]{align*}
    \boxed{\begin{matrix} 
    + & + & + \\ 
    \cdot & \cdot &\cdot\\
    \cdot & \cdot & \cdot
    \end{matrix}}
    \boxed{\begin{matrix} + & + & + \\ 
    \cdot & \cdot &\cdot\\
    \cdot & \cdot & \cdot
    \end{matrix}} \ \ \ 
    \boxed{\begin{matrix} 
    + & \cdot & \cdot \\ 
    + & \cdot &\cdot\\
    + & + & \cdot
    \end{matrix}}
    \boxed{\begin{matrix} + & \cdot & \cdot \\ 
    + & \cdot &\cdot\\
    \cdot & \cdot & \cdot
    \end{matrix}} \ &\ \
    \boxed{\begin{matrix} 
    + & \cdot & \cdot \\ 
    + & \cdot &\cdot\\
    + & \cdot & \cdot
    \end{matrix}}
    \boxed{\begin{matrix} + & \cdot & + \\ 
    + & \cdot &\cdot\\
    \cdot & \cdot & \cdot
    \end{matrix}} \ \ \ 
    \boxed{\begin{matrix} 
    + & \cdot & \cdot \\ 
    + & \cdot &\cdot\\
    + & \cdot & \cdot
    \end{matrix}}
    \boxed{\begin{matrix} + & \cdot & \cdot \\ 
    + & + & \cdot\\
    \cdot & \cdot & \cdot
    \end{matrix}}\\
    \boxed{\begin{matrix} 
    \cdot & + & + \\ 
    + & \cdot & + \\
    + & + & \cdot
    \end{matrix}}
    \boxed{\begin{matrix} \cdot & \cdot & \cdot \\ 
    \cdot & \cdot & \cdot\\
    \cdot & \cdot & \cdot
    \end{matrix}} \ \ \ 
    \boxed{\begin{matrix} 
    \cdot & + & + \\ 
    + & \cdot & \cdot \\
    + & + & \cdot
    \end{matrix}}
    \boxed{\begin{matrix} \cdot & + & \cdot \\ 
    \cdot & \cdot & \cdot\\
    \cdot & \cdot & \cdot
    \end{matrix}} \ &\ \
    \boxed{\begin{matrix} 
    \cdot & + & + \\ 
    + & \cdot & \cdot \\
    + & + & \cdot
    \end{matrix}}
    \boxed{\begin{matrix} \cdot & \cdot & \cdot \\ 
    + & \cdot & \cdot\\
    \cdot & \cdot & \cdot
    \end{matrix}}  \ \ \ 
    \boxed{\begin{matrix} 
    \cdot & + & + \\ 
    + & \cdot & \cdot \\
    + & \cdot & \cdot
    \end{matrix}}
    \boxed{\begin{matrix} \cdot & + & + \\ 
    \cdot & \cdot & \cdot\\
    \cdot & \cdot & \cdot
    \end{matrix}}\\
    \boxed{\begin{matrix} 
    \cdot & + & + \\ 
    + & \cdot & \cdot \\
    + & \cdot & \cdot
    \end{matrix}}
    \boxed{\begin{matrix} \cdot & \cdot & + \\ 
    + & \cdot & \cdot\\
    \cdot & \cdot & \cdot
    \end{matrix}} \ \ \ 
    \boxed{\begin{matrix} 
    \cdot & + & + \\ 
    + & \cdot & \cdot \\
    + & \cdot & \cdot
    \end{matrix}}
    \boxed{\begin{matrix} \cdot & \cdot & \cdot \\ 
    + & + & \cdot\\
    \cdot & \cdot & \cdot
    \end{matrix}} \ &\ \
    \boxed{\begin{matrix} 
    \cdot & + & + \\ 
    \cdot & \cdot & \cdot \\
    + & \cdot & \cdot
    \end{matrix}}
    \boxed{\begin{matrix} + & + & + \\ 
    \cdot & \cdot & \cdot\\
    \cdot & \cdot & \cdot
    \end{matrix}} \ \ \ 
    \boxed{\begin{matrix} 
    \cdot & + & + \\ 
    \cdot & \cdot & \cdot \\
    \cdot & \cdot & \cdot
    \end{matrix}}
    \boxed{\begin{matrix} + & + & + \\ 
    + & \cdot & \cdot\\
    \cdot & \cdot & \cdot
    \end{matrix}}\\
    \boxed{\begin{matrix} 
    \cdot & + & \cdot \\ 
    + & \cdot & \cdot \\
    \cdot & \cdot & \cdot
    \end{matrix}}
    \boxed{\begin{matrix} + & + & + \\ 
    + & \cdot & \cdot\\
    \cdot & \cdot & \cdot
    \end{matrix}} \ \  \
    \boxed{\begin{matrix} 
    \cdot & + & \cdot \\ 
    \cdot & \cdot & \cdot \\
    \cdot & \cdot & \cdot
    \end{matrix}}
    \boxed{\begin{matrix} + & + & + \\ 
    + & + & \cdot\\
    \cdot & \cdot & \cdot
    \end{matrix}}  \ &\ \ 
    \boxed{\begin{matrix} 
    \cdot & \cdot & + \\ 
    + & \cdot & + \\
    + & + & \cdot
    \end{matrix}}
    \boxed{\begin{matrix} + & \cdot & \cdot \\ 
    \cdot & \cdot & \cdot\\
    \cdot & \cdot & \cdot
    \end{matrix}} \ \ \ 
    \boxed{\begin{matrix} 
    \cdot & \cdot & + \\ 
    + & \cdot & \cdot \\
    + & + & \cdot
    \end{matrix}}
    \boxed{\begin{matrix} + & + & \cdot \\ 
    \cdot & \cdot & \cdot\\
    \cdot & \cdot & \cdot
    \end{matrix}}\\
    \boxed{\begin{matrix} 
    \cdot & \cdot & + \\ 
    + & \cdot & \cdot \\
    + & + & \cdot
    \end{matrix}}
    \boxed{\begin{matrix} + & \cdot & \cdot \\ 
    + & \cdot & \cdot\\
    \cdot & \cdot & \cdot
    \end{matrix}} \ \ \ 
    \boxed{\begin{matrix} 
    \cdot & \cdot & + \\ 
    + & \cdot & \cdot \\
    + & \cdot & \cdot
    \end{matrix}}
    \boxed{\begin{matrix} + & \cdot & + \\ 
    + & \cdot & \cdot\\
    \cdot & \cdot & \cdot
    \end{matrix}} \ &\ \ 
    \boxed{\begin{matrix} 
    \cdot & \cdot & + \\ 
    + & \cdot & \cdot \\
    + & \cdot & \cdot
    \end{matrix}}
    \boxed{\begin{matrix} + & \cdot & \cdot \\ 
    + & + & \cdot\\
    \cdot & \cdot & \cdot
    \end{matrix}} \ \ \ 
    \boxed{\begin{matrix} 
    \cdot & \cdot & + \\ 
    \cdot & \cdot & \cdot \\
    + & + & \cdot
    \end{matrix}}
    \boxed{\begin{matrix} + & + & + \\ 
    \cdot & \cdot & \cdot\\
    \cdot & \cdot & \cdot
    \end{matrix}}\\
    \boxed{\begin{matrix} 
    \cdot & \cdot & + \\ 
    \cdot & \cdot & \cdot \\
    + & \cdot & \cdot
    \end{matrix}}
    \boxed{\begin{matrix} + & + & + \\ 
    + & \cdot & \cdot\\
    \cdot & \cdot & \cdot
    \end{matrix}} \ \ \ 
    \boxed{\begin{matrix} 
    \cdot & \cdot & \oplus \\ 
    \cdot & \cdot & \cdot \\
    + & \cdot & \cdot
    \end{matrix}}
    \boxed{\begin{matrix} + & + & + \\ 
    + & \cdot & \cdot\\
    \cdot & \cdot & \cdot
    \end{matrix}} \ &\ \ 
    \boxed{\begin{matrix} 
    \cdot & \cdot & \cdot \\ 
    + & + & \cdot \\
    + & + & \cdot
    \end{matrix}}
    \boxed{\begin{matrix} + & + & \cdot \\ 
    \cdot & \cdot & \cdot\\
    \cdot & \cdot & \cdot
    \end{matrix}} \ \ \ 
    \boxed{\begin{matrix} 
    \cdot & \cdot & \cdot \\ 
    + & \cdot & \cdot \\
    + & + & \cdot
    \end{matrix}}
    \boxed{\begin{matrix} + & + & \cdot \\ 
    + & \cdot & \cdot\\
    \cdot & \cdot & \cdot
    \end{matrix}}\\
\boxed{\begin{matrix} 
    \cdot & \cdot & \cdot \\ 
    + & \cdot & \cdot \\
    + & \cdot & \cdot
    \end{matrix}}
    \boxed{\begin{matrix} + & + & + \\ 
    + & \cdot & \cdot\\
    \cdot & \cdot & \cdot
    \end{matrix}} \ \ \ 
    \boxed{\begin{matrix} 
    \cdot & \cdot & \cdot \\ 
    + & \cdot & \cdot \\
    + & \cdot & \cdot
    \end{matrix}}
    \boxed{\begin{matrix} + & + & \cdot \\ 
    + & + & \cdot\\
    \cdot & \cdot & \cdot
    \end{matrix}} \ &\ \ 
    \boxed{\begin{matrix} 
    \cdot & \cdot & \cdot \\ 
    \cdot & + & \cdot \\
    + & + & \cdot
    \end{matrix}}
    \boxed{\begin{matrix} + & + & + \\ 
    \cdot & \cdot & \cdot\\
    \cdot & \cdot & \cdot
    \end{matrix}}  \ \ \ 
    \boxed{\begin{matrix} 
    \cdot & \cdot & \cdot \\ 
    \cdot & \cdot & \cdot \\
    + & + & \cdot
    \end{matrix}}
    \boxed{\begin{matrix} + & + & + \\ 
    + & \cdot & \cdot\\
    \cdot & \cdot & \cdot
    \end{matrix}}\\
    \boxed{\begin{matrix} 
    \cdot & \cdot & \cdot \\ 
    \cdot & \cdot & \cdot \\
    + & \cdot & \cdot
    \end{matrix}}
    \boxed{\begin{matrix} + & + & + \\ 
    + & + & \cdot\\
    \cdot & \cdot & \cdot
    \end{matrix}} \ \ \ 
    \boxed{\begin{matrix} 
    \cdot & \cdot & \cdot \\ 
    \cdot & \cdot & \cdot \\
    \oplus & \cdot & \cdot
    \end{matrix}}
    &\boxed{\begin{matrix} + & + & + \\ 
    + & + & \cdot\\
    \cdot & \cdot & \cdot
    \end{matrix}} \ \ \ 
    \boxed{\begin{matrix} 
    \cdot & \cdot & \cdot \\ 
    \cdot & \cdot & \cdot \\
    \cdot & \cdot & \cdot
    \end{matrix}}
    \boxed{\begin{matrix} + & + & + \\ 
    + & + & +\\
    \cdot & \cdot & \cdot
    \end{matrix}} \ \ \
\end{empheq}
    
The additional component of lower dimension corresponds to
$\boxed{\begin{matrix} 
\cdot & \cdot & + \\ 
\cdot & \cdot & + \\
\oplus & \cdot & \cdot
\end{matrix}}
\boxed{\begin{matrix} + & + & + \\ 
+ & \cdot & \cdot \\
\cdot & \cdot & \cdot
\end{matrix}}$.
The formula of Knutson--Zinn-Justin expresses
$31=1+2+2+2+4+4+8+8$ as a weighted sum of eight \emph{generic pipe dreams}. This numerical pattern,
along with our usage of $\oplus$, leads to the suspicion
that the $n=3$ tablet fundamentally differs from their rule.
How does one decipher the tablets for $n=2,3$ above (or find an alternate tablet with an interpretation)?

\section{Hilbert-Samuel multiplicities, especially of Schubert varieties}\label{sec:HS}

\subsection{Hilbert-Samuel multiplicities} \label{sec:HSmult} 
The \emph{projectivized tangent cone} $TC_p({\mathfrak X})$ at $p={\bf 0}$ to ${\mathfrak X}=V(J)$ is the projective variety of ${\mathbb P}^{n-1}$ defined by the \emph{tangent cone ideal}, which is the
homogeneous ideal $I$
generated by the lowest degree \emph{forms} of every $f\in J$. The \emph{Hilbert-Samuel multiplicity} of 
${\mathfrak X}$ at $p={\bf 0}$ is the degree of $TC_p({\mathfrak X})$ in ${\mathbb P}^{n-1}$. 

Suppose ${\mathfrak X}$ is an arbitrary variety and $p\in {\mathfrak X}$. The projectivized tangent cone and the  Hilbert-Samuel multiplicity of ${\mathfrak X}$ at $p$ is defined by first picking an affine open neighborhood around $p$ with coordinates
where $p$ is ${\bf 0}$. Suppose this neighborhood is the affine
variety defined by $R/J$. Then the \emph{Hilbert-Samuel multiplicity} of ${\mathfrak X}$ at $p$, denoted ${\mathrm{ mult}}_p({\mathfrak X})$, is the Hilbert-Samuel multiplicity of $V(J)$ as defined in the previous paragraph. 

The size of ${\mathrm{ mult}}_p({\mathfrak X})$ measures ``how 
singular'' ${\mathfrak X}$ is at $p$. For instance, ${\mathrm{ mult}}_p({\mathfrak X})=1 \iff p$ is smooth in ${\mathfrak X}$. For many subvarieties of the complete flag variety $GL_n({\mathbb C})/B$ one seeks
a counting rule for ${\mathrm{mult}}_p({\mathfrak X})$; see, e.g., 
work on \emph{Peterson varieties} \cite[Section~7]{Insko} and \emph{spherical symmetric orbit closures} \cite[Section~7]{WWY}. However, it is for \textit{Schubert varieties} that
a multiplicity rule has long been sought after.

\subsection{Schubert varieties}
Let $B$ be the Borel subgroup of invertible upper triangular matrices in $GL_n({\mathbb C})$. $B$ acts on $GL_n({\mathbb C})/B$ with finitely many orbits $BwB/B$ where $w\in S_n$ is viewed as a permutation matrix. Their closures 
$X_w=\overline{BwB/B}$
are the \emph{Schubert varieties}. Let $T\subset GL_n({\mathbb C})$ be the subgroup of invertible diagonal matrices. $GL_n/B$ has $n!$ fixed points $p_v=vB/B$, one for each $v\in S_n$. Such a fixed point appears in $X_w$ if and only if $v\leq w$ in Bruhat order. We refer to the survey \cite{WY} for more about Schubert varieties in the conventions used here. 

In \cite{WY:Govern, WY:Grob, WY}, one uses \emph{Kazhdan-Lusztig varieties}
${\mathcal N}_{v,w}$ with a specific choice of coordinates and equations to define (up to a factor of affine space) a local neighborhood of $p_v\in X_w$.
That is ${\mathcal N}_{v,w}={\mathrm{Spec}}(R/I_{v,w})$ 
for the \emph{Kazhdan-Lusztig ideal} $I_{v,w}$.\footnote{These equations have been implemented in the {\tt Macaualay2} code \url{https://faculty.math.illinois.edu/~ayong/Schubsingular.v0.2.m2}.}
Here $R$ is the coordinate ring of the matrix $Z_{v}$ defined as follows. First place  ``$1$'' in column $j$ and row $v(j)$ (from the top).
Next, set all matrix positions right and above a ``$1$'' to $0$. Finally, in the remaining spots, put  $x_{ij}$ in row $i$ from the \emph{bottom} and column $j$ from the left. Hence, the multiplicity problem for $X_w$ falls into Theorem~\ref{thm:main}'s purview, where $I\subset R$ is the tangent cone ideal of $I_{v,w}$.

The paper \cite[Theorem~6.1]{LiLiYong} gives a rule for ${\mathrm{ mult}}_{p_v}(X_w)$
whenever $w$ is \emph{covexillary} ($3412$-avoiding). It produces a term order 
$\prec$ on $R$ such that $J={\mathrm{init}}_{\prec} I$ matches the initial ideals for vexillary matrix Schubert varieties studied in
\cite{KMY}.\footnote{The use of covexillary in the former and vexillary in the latter comes from a difference in convention. The interesting point is the use of comparisons of tablets for \emph{different problems}.}   This order depends on $v$ and $w$.
Now, \cite[Conjecture~8.1]{LiLiYong} conjectures that under a ``SE-NW'' term order that orders the variables from reading columns right to left and bottom to top, the limit is  reduced. 

\begin{example}
Let us pick this ``SE-NW'' order; by virtue of Theorem~\ref{thm:main} the fact that we cannot \emph{prove} the limit is reduced
is moot. We now
compute the tablet for a multiplicity problem which is not covered by any combinatorial rule we know of. One such example is
$w=463512$ and $v=id$. \footnote{$(id)B/B$ is the point with the largest multiplicity in $X_w$ and in some sense, the ``worst case''.}
Here, $Z_v$ is simply a lower triangular unipotent matrix:
\[Z_v=\left(\begin{matrix}
1 & 0 & 0 & 0 & 0 & 0\\
x_{51} & 1 & 0 & 0 & 0 & 0\\
x_{41} & x_{42} & 1 & 0 & 0 & 0\\
x_{31} & x_{32} & x_{33} & 1 & 0 & 0\\
x_{21} & x_{22} & x_{23} & x_{24} & 1 & 0\\
x_{11} & x_{12} & x_{13} & x_{14} & x_{15} & 1
\end{matrix}\right).\]
%
The choice of $w=463512$ defines $I_{v,w}$ to be generated by $x_{11},x_{21}$,
the $2\times 2$ minors of the southwest
$2\times 3$ submatrix of $Z_v$, and the
southwest $5\times 5$ minor.
The tangent cone is
\[I= \langle x_{21}, x_{11}, x_{13}x_{22} - x_{23}x_{12}, x_{14}x_{31} + x_{13}x_{41} 
+ x_{12}x_{51}\rangle.\]
The initial ideal is squarefree so
$\widetilde{J}=J=\langle x_{21}, x_{11}, x_{13}x_{22}, 
x_{14}x_{31}\rangle$.
It is also true that $\widetilde{J}$ is equidimensional. Thus the four hieroglyphs below encode the components
with the natural coordinates:
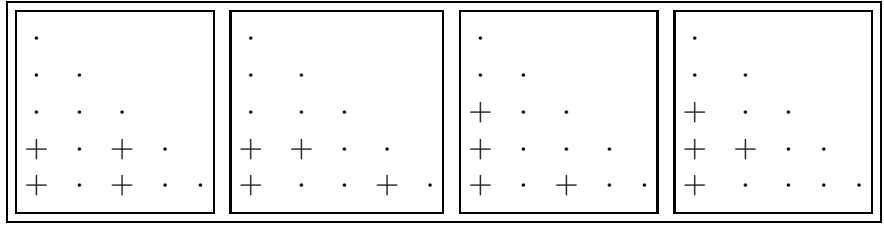
\[\boxed{
\boxed{\begin{matrix}
\cdot & \ & \ & \ & \ \\
\cdot & \cdot & \ & \ & \ \\
\cdot & \cdot & \cdot & \ & \ \\
+ & \cdot & + & \cdot & \ \\
+ & \cdot & + & \cdot & \cdot \\
\end{matrix}} \ \ 
\boxed{\begin{matrix}
\cdot & \ & \ & \ & \ \\
\cdot & \cdot & \ & \ & \ \\
\cdot & \cdot & \cdot & \ & \ \\
+ & + & \cdot & \cdot & \ \\
+ & \cdot & \cdot & + & \cdot \\
\end{matrix}} \ \ 
\boxed{\begin{matrix}
\cdot & \ & \ & \ & \ \\
\cdot & \cdot & \ & \ & \ \\
+ & \cdot & \cdot & \ & \ \\
+ & \cdot & \cdot & \cdot & \ \\
+ & \cdot & + & \cdot & \cdot \\
\end{matrix}} \ \ 
\boxed{\begin{matrix}
\cdot & \ & \ & \ & \ \\
\cdot & \cdot & \ & \ & \ \\
+ & \cdot & \cdot & \ & \ \\
+ & + & \cdot & \cdot & \ \\
+ & \cdot & \cdot & \cdot & \cdot \\
\end{matrix}} }
\]
That is, ${\mathrm{mult}}_{p_{id}}(X_{463512})=4$.\qed
%
\end{example}
Theorem~\ref{thm:main} indicates the existence of a general unweighted rule for ${\mathrm{mult}}_{p_v}(X_w)$ in a form similar to that of \cite[Theorem~6.1]{LiLiYong}, but with respect to a universal term order rather than one that depends on $v$ or $w$. Intriguingly, the production of the tablets poses no problem -- there is ample data to study. However, the combinatorial challenge of succinctly describing the tablets generated by the combinatorial commutative algebra rule still persists.
\section*{Acknowledgements}
AY thanks  Shiliang Gao, Allen Knutson, Li Li, Ezra Miller, Anna Weigandt, and Alexander Woo from whom he learned about Gr\"obner geometry through collaborations. We thank Avery St. Dizier for sharing his code to compute bumpless pipe dreams. We made use of {\tt Macaulay2} to execute our
computations.
AS was supported by a Susan C.~Morosato IGL graduate student scholarship. 
AY was supported by a Simons Collaboration grant. Both authors were partially supported by an NSF RTG in Combinatorics (DMS 1937241).

\end{document}